\documentclass[10pt,fleqn,leqno]{siamltex}

\usepackage[english]{babel}
\usepackage{amsmath,amssymb,amsfonts,graphicx,subfigure,url}
\usepackage{myshortcuts}
\usepackage[T1]{fontenc}
\usepackage[font=small,labelfont=bf,tableposition=top]{caption}
\DeclareCaptionLabelFormat{andtable}{#1~#2  \&  \tablename~\thetable}

\newcommand{\bN}{\mathbf{N}}

\newcommand{\comp}{\rm comp}
\newtheorem{remark}[theorem]{Remark}
\date{\today}
\title{The infinite bi-Lanczos method for nonlinear eigenvalue problems \thanks{
Version \today.}}

\author{Sarah W.~Gaaf\thanks{Department of Mathematics and Computer Science,
TU Eindhoven,
PO Box 513,
5600 MB Eindhoven,
The Netherlands,
\texttt{s.w.gaaf@tue.nl}.
This author is supported by a Vidi research grant from the
Netherlands Organisation for Scientific Research (NWO).}
\and Elias Jarlebring\thanks{Department of Mathematics, Royal Institute of Technology (KTH),
Stockholm, SeRC Swedish e-Science Research Center, \texttt{eliasj@kth.se}.
}
}

\begin{document}

\maketitle
\begin{abstract}
We propose a two-sided Lanczos method for the nonlinear eigenvalue problem (NEP).
This two-sided approach provides approximations to both the right and
left eigenvectors of the eigenvalues of interest.
The method implicitly works with matrices and vectors with infinite size,
but because particular (starting) vectors are used, all computations
can be carried out efficiently with finite matrices and vectors.
We specifically introduce a new way to represent infinite vectors that span the subspace corresponding to the conjugate transpose operation for approximating the left eigenvectors.
Furthermore, we show that also in this infinite-dimensional interpretation the short recurrences inherent to the Lanczos procedure offer an efficient
algorithm regarding both the computational cost and the storage.
\end{abstract}



\section{Introduction}\label{sec:intro}
Let $M:\CC\rightarrow\CC^{n\times n}$ be a matrix
depending on a parameter with elements that are analytic in $\rho\bar\DD$,
where $\rho>0$ a constant, $\DD$ is the open unit disk and $\bar\DD$ its closure.
We present a new method for the nonlinear eigenvalue problem:
find $(\lambda,x,y)\in \rho\DD\times\CC^n\times\CC^n$, where $x \neq 0$, $y\neq 0$, such that
\begin{subequations}\label{eq:nep}
\begin{align}
M(\lambda) \, x &= 0\label{eq:nepa}\\
M(\lambda)^*y &= 0\label{eq:nepb}.
\end{align}
\end{subequations}
We are interested in both the left and the right eigenvectors of the problem.
The simultaneous approximation of both left and right eigenvectors
is useful, e.g., in the estimation of the eigenvalue condition number
and the vectors can be used as initial values for locally convergent two-sided iterative methods, e.g., those described in \cite{Schreiber:2008:PHD}.
The NEP \eqref{eq:nep} has received considerable attention
in the numerical linear algebra community, and there are several
competitive numerical methods. There are for instance,
so-called single vector methods such as Newton type methods \cite{Schreiber:2008:PHD,Szyld:2012:LOCAL,Effenberger:2013:ROBUST},
which often can be improved with subspace acceleration, see \cite{Voss:2004:ARNOLDI}, and
Jacobi--Davidson methods \cite{Betcke:2004:JD}. These have been extended in a block sense \cite{Kressner:2009:BLOCKNEWTON}.
There are methods specialized for symmetric problems that
have an (easily computed) Rayleigh functional \cite{Szyld:2016:PRECOND_I}.
There is also a recent class of methods which can be interpreted as either dynamically
extending an approximation or carrying out an infinite-dimensional
algorithm, see for instance  \cite{Jarlebring:2012:INFARNOLDI,VanBeeumen:2013:RATIONAL,Guttel:2014:NLEIGS} and
references therein.
For recent developments see the summary papers \cite{Ruhe:1973:NLEVP,Mehrmann:2004:NLEVP,Voss:2013:NEPCHAPTER} and
the benchmark collection \cite{Betcke:2013:NLEVPCOLL}.

We propose a new method that is based on the two-sided Lanczos method for non-Hermitian problems.
An intuitive derivation of the main idea of this paper is the following.
Suppose $(\lambda, x)$ is a solution to \eqref{eq:nepa}.
By adding trivial identities we have an equality
between vectors of infinite length (cf. \cite{Jarlebring:2012:INFARNOLDI})

\begin{equation}\label{eq:first_infdim_eig0}
{\scriptstyle 
  \begin{bmatrix}
    -M(0)&&& \\
       & I && \\
       && I & \\
       &&& \ddots
  \end{bmatrix}
  \begin{bmatrix}
    \tfrac{\lambda^0}{0!}x\\[0.5mm]
    \tfrac{\lambda^1}{1!}x\\[0.5mm]
    \tfrac{\lambda^2}{2!}x\\[0.5mm]
  \vdots
  \end{bmatrix}
=
  \lambda
\begin{bmatrix}
    \tfrac{1}{1}M'(0)&\tfrac{1}{2}M''(0)& \tfrac{1}{3}M^{(3)}(0)& \cdots\\
    \tfrac{1}{1}I & & & \\
    & \tfrac{1}{2}I & &  \\
    & & \tfrac{1}{3}I &  \\
    & & & \ddots  \\
\end{bmatrix}
  \begin{bmatrix}
    \tfrac{\lambda^0}{0!}x\\[0.5mm]
    \tfrac{\lambda^1}{1!}x\\[0.5mm]
    \tfrac{\lambda^2}{2!}x\\[0.5mm]
  \vdots
  \end{bmatrix}
  }
  .
\end{equation}
Here, $I$ is the $n \times n$ identity matrix. One variant 
of the infinite Arnoldi method \cite{Jarlebring:2012:INFARNOLDI}
is based on carrying out Arnoldi's method on the infinite-dimensional
system \eqref{eq:first_infdim_eig0}. Our approach here is based on the
two-sided Lanczos and requires analysis also of the transposed matrix.
Throughout the paper we assume that $0$ is not an eigenvalue, so that $M(0)^{-1}$ exists.
(This does not represent a loss of generality, as we can apply
a shift in case 0 is an eigenvalue.)
Let $\bN \in\CC^{n\times\infty}$ be defined by
\begin{align*}
\bN & := \begin{bmatrix} N_1 & N_2 & N_3 & \cdots \ \end{bmatrix} \\
& := \begin{bmatrix}
-M(0)^{-1} M'(0) & -\tfrac{1}{2} M(0)^{-1} M''(0) & -\tfrac{1}{3} M(0)^{-1} M^{(3)}(0) & \cdots \
\end{bmatrix}
\end{align*}
and define a vector of infinite length
$\bv:=[v_j]_{j=1}^\infty=[\tfrac{\lambda^{(j-1)}}{(j-1)!}x]_{j=1}^\infty$, where $v_j \in \mathbb{C}^n$ for $j = 1,2, \ldots$.
Relation \eqref{eq:first_infdim_eig0} can now
be more compactly expressed as
\begin{equation}\label{eq:first_infdim_eig}
\bv=\lambda \, (\be_1\otimes \bN+\bS\otimes I) \, \bv, \qquad \text{where}
\quad
\bS := \left[\begin{array}{cccc} 0 & 0 & 0 & \cdots \\ \tfrac{1}{1} &&& \\ & \tfrac{1}{2} && \\ && \tfrac{1}{3} & \\ &&&\ddots \end{array}\right],
\end{equation}
and $\be_1 = \begin{bmatrix} 1 \ 0 \ 0 \ \cdots \ \end{bmatrix}^T$ is the first basis vector.
Equations~\eqref{eq:first_infdim_eig0} and \eqref{eq:first_infdim_eig} may be viewed as a
companion linearization for the nonlinear eigenvalue problem.
Note that a solution $\lambda$ to \eqref{eq:first_infdim_eig} corresponds
to a reciprocal eigenvalue of the infinite-dimensional matrix
\begin{equation}
\label{eq:bA_def}
\bA := \be_1\otimes \bN +\bS\otimes I.
\end{equation}
The derivation of our new two-sided Lanczos procedure is based on applying
the Lanczos method (for non-Hermitian problems) to the infinite-dimensional matrix $\bA$.
The method builds two bi-orthogonal subspaces using short recurrences.
One subspace serves the approximation of right eigenvectors,
the other the approximation of the left eigenvectors.
In Section~\ref{sec:infdim} we use that, analogously to companion linearizations for polynomial
eigenvalue problems, relation \eqref{eq:first_infdim_eig} is equivalent to \eqref{eq:nepa},
which has been used in \cite{Jarlebring:2012:INFARNOLDI}. 
For the approximation of solutions to \eqref{eq:nepb} we derive a new and more involved
relationship for the left eigenvectors, also presented in Section~\ref{sec:infdim}.
This leads to a new way to represent infinite vectors that span the subspace corresponding to the conjugate transpose operation for approximating the left eigenvectors.
With two particular types of (starting) vectors, we can carry out an algorithm for the infinite-dimensional operator $\bA$ using only finite arithmetic.
This is covered in the first four subsections of Section~\ref{sec:bilanczos}, where we also treat the computation of
scalar products and matrix-vector products for the infinite-dimensional case. 
The second half of Section~\ref{sec:bilanczos} is dedicated to various computational issues and complexity considerations.
In Section~\ref{sec:numexpes} we present a few examples to illustrate the performance of the new method,
and we conclude with a short discussion.

Throughout this paper we use bold symbols to indicate matrices or vectors of
infinite dimensions, i.e., an infinite matrix is denoted by $\bA\in\CC^{\infty\times\infty}$,
and an infinite-dimensional vector is denoted by $\bx \in\CC^{\infty}$.
Unless otherwise stated, the $n$-length blocks of a vector of infinite length are denoted with subscript,
e.g., $\bw=[w_1^T,w_2^T,\ldots \ ]^T$ where $w_j\in\CC^n$ for $j\ge1$.

\section{Infinite dimensional reformulation}\label{sec:infdim}

In our formalization of the operator $\bA$ we first need to define its domain.
This is necessary to prove equivalence between $(\lambda,x,y)$
which is a solution to \eqref{eq:nep} and the eigentriplet $(\mu, \bv, \bw)$ of $\bA$,
where  $\mu = \lambda^{-1}$.
Let $\|\cdot\|$ denote the 2-norm. It will turn out to be natural to define the operators on a weighted, mixed 1-norm and 2-norm space defined by
\begin{align}
  \ell_{1}(\rho)&:=\Big\{\bw=[w_j]_{j=1}^\infty \in\CC^\infty:
\sum_{j=1}^\infty \tfrac{\rho^j}{j!} \, \|w_j\|
< \infty\Big\}.
\end{align}
Note that some vectors in $\ell_1(\rho)$ correspond to sequences
of vectors which are unbounded, i.e., $\|w_j\|\rightarrow\infty$ as $j\rightarrow\infty$,
but do not grow arbitrarily fast, since $\bw\in\ell_1(\rho)$ implies that
\begin{equation}  \label{eq:growbound}
  \tfrac{\rho^j}{j!} \, \|w_j\|\rightarrow 0 \quad \textrm{ as }j\rightarrow\infty.
\end{equation}
In the proofs of the theorems and propositions below we need to allow the vectors to have this growth,
to accommodate the fact that derivatives of analytic functions are not necessarily bounded.
We will let $\rho$ be the convergence radius of the power series expansion of the
analytic function $M$, and set
$
\DDD(\bA)=\DDD(\bA^*)=\ell_1(\rho)
$
as the domain of the operator.
The following two theorems do not only show the equivalence between the nonlinear
eigenvalue problem and the operator $\bA$, but also reveal the structure
of the left and right eigenvectors of $\bA$.
The first result is an adaption of \cite[Thm.~1]{Jarlebring:2012:INFARNOLDI} for our discrete operator and only assuming a finite convergence radius.
\begin{theorem}[Right eigenvectors of $\bA$]\label{thm:righteigvec}
Suppose $M$ is analytic in $\lambda\in \rho\bar\DD$
and let $\bA$ be defined by \eqref{eq:bA_def}.
\begin{itemize}
  \item[(i)] If $(\mu,\bv)\in \CC\times\DDD(\bA)\backslash\{0\}$ is an eigenpair of $\bA$ and $\lambda=\mu^{-1}\in \rho\DD$,
then there exists a vector $x\in\CC^n$ such that
\begin{equation}  \label{eq:bv_equiv}
\bv=\left[\tfrac{\lambda^{j-1}}{(j-1)!} \, x\right]_{j=1}^\infty.
\end{equation}
  \item[(ii)] The pair $(\lambda,x)\in \rho\DD\backslash\{0\}\times \CC^n\backslash\{0\}$
is a solution to \eqref{eq:nepa}
if and only if the pair
$\left(\lambda^{-1},\bv\right)\in
(\CC\backslash \rho^{-1}\bar\DD) \times \DDD(\bA)$
is an eigenpair of $\bA$, where $\bv$ is given by \eqref{eq:bv_equiv}.
\end{itemize}
\end{theorem}
\begin{proof}
To show (i), let $\bv=[v_j]_{j=1}^\infty$, where $v_j\in\CC^n$ are the blocks of $\bv$.
From the block rows $2,3,\ldots$ of $\lambda \bA \bv=\bv$, we have that $v_{j+1}=\tfrac{\lambda}{j} \, v_j$ for $j=1,2,\ldots$.
It follows from induction that the blocks in the eigenvector satisfy
\begin{equation}
  v_j=\tfrac{\lambda^{j-1}}{(j-1)!} \, v_1, \;\;j=1,2,\ldots.
\end{equation}
We also have that $\bv\in\ell_1(\rho)$ since $\|v_j\|=\tfrac{|\lambda|^{j-1}}{(j-1)!} \, \|v_1\|$
such that $\bv\in\ell_1\subset\ell_1(\rho)$.

To show (ii), assume first that $(1/\lambda,\bv)$ is an eigenpair of $\bA$.
From (i) we know that the blocks of $\bv$ satisfy $v_j=\tfrac{\lambda^{j-1}}{(j-1)!} \, v_1$.
The first block row of $\bv=\lambda\bA \bv$ implies that
\[
v_1
= \lambda \, \sum_{j=1}^{\infty} \, N_j \, \tfrac{\lambda^{j-1}}{(j-1)!} \, v_1
= -\sum_{j=1}^{\infty}\tfrac{\lambda^j}{j!} M(0)^{-1} \, M^{(j)}(0) \, v_1
= -M(0)^{-1} \left(M(\lambda)-M(0)\right)v_1.
\]
Therefore, since $0$ is not an eigenvalue, $(\lambda,v_1)$ is a solution to \eqref{eq:nepa}.

To show the converse, suppose that $(\lambda,x)$ is a solution to \eqref{eq:nepa}.
Let $\bv$ be as in \eqref{eq:bv_equiv}. The rest of the proof consists of showing that
\begin{equation}\label{eq:proof2}
\lambda \bA \bv = \bv.
\end{equation}
Similar to above, the first block row of $\lambda \bA \bv$ is
\[
-\lambda \, \displaystyle \sum_{j=1}^{\infty} \tfrac{1}{j} \, M(0)^{-1} \, M^{(j)}(0) \, v_j
    = -M(0)^{-1}M(\lambda) \, x + x = x.
\]
In the last step we used that $M(\lambda)x=0$,
since $(\lambda,x)$ is a solution to \eqref{eq:nepa}.
Hence, the equality in the first block row of \eqref{eq:proof2} is proven.
The equality in \eqref{eq:proof2} corresponding to blocks $j>1$ follows from
the fact that $v_{j+1}=(\lambda/j) \, v_j$, $j = 1, 2, \dots$, by construction.
\end{proof}

We now study the equivalence between a left eigenpair of the nonlinear eigenvalue problem
and a left eigenpair of $\bA$. Also, the structure of the left eigenvectors of $\bA$ will be concretized.

\begin{theorem}[Left eigenvectors of $\bA$]\label{thm:lefteigvec}
Suppose $M$ is analytic in $\lambda\in \rho\bar\DD$
and let $\bA^*$ be defined by \eqref{eq:bA_def}.
\begin{itemize}
  \item[(i)] If $(\mu,\bw)\in \CC\times\DDD(\bA^*)\backslash\{0\}$ is an eigenpair of $\bA^*$ and $\lambda=\mu^{-1}\in \rho\DD$,
then there exists a vector $z\in\CC^n$ such that
\begin{equation}  \label{eq:bw_equiv}
\bw=\sum_{j=1}^\infty (\bS^T\otimes I)^{j-1}\bN^*\lambda^j z.
\end{equation}
  \item[(ii)] The pair $(\lambda,y)\in \rho\DD\backslash\{0\}\times \CC^n\backslash\{0\}$
is a solution to \eqref{eq:nepb} if and only if the pair
$\left(\lambda^{-1},\bw\right)\in (\CC\backslash \rho^{-1}\bar\DD) \times \DDD(\bA^*)$ is an eigenpair
of $\bA^*$, where $\bw$ is given by \eqref{eq:bw_equiv} with $z=M(0)^*y$.
\end{itemize}
\end{theorem}

\begin{proof}
Suppose $\lambda\bA^*\bw = \bw$, where $\bw\in\ell_1(\rho)$.
We use induction to show that
\begin{equation}\label{eq:w1_in_proof0}
w_1 = \sum_{j=1}^{k} \tfrac{\lambda^j}{(j-1)!}  N_j^* \, w_1 + \tfrac{\lambda^k}{k!} \, w_{k+1}
\end{equation}
for any $k$.
Relation~\eqref{eq:w1_in_proof0} is easy to see for $k=1$.
Suppose \eqref{eq:w1_in_proof0} is satisfied for $k-1$, i.e.,
\begin{equation}\label{eq:w1_in_proofN-1}
w_1 = \sum_{j=1}^{k-1} \tfrac{\lambda^j}{(j-1)!} N_j^* \, w_1 + \tfrac{\lambda^{k-1}}{(k-1)!} \, w_{k}.
\end{equation}
Block row $k$ of $\lambda\bA^*\bw = \bw$ reduces to
\begin{equation}\label{eq:Nth_row_in_proof}
\lambda \, N_k^* \, w_1 + \tfrac{\lambda}{k} \, w_{k+1} = w_k.
\end{equation}
The induction is completed by inserting \eqref{eq:Nth_row_in_proof} in
relation \eqref{eq:w1_in_proofN-1}, which yields \eqref{eq:w1_in_proof0}.
Due to the fact that $\bw\in\ell_1(\rho)$, \eqref{eq:growbound}
holds, and since $|\lambda|<\rho$, we have $\|\tfrac{\lambda^{k}}{k!}w_{k+1} \|<
\tfrac{\rho^k}{k!} \, \|w_{k+1}\|\rightarrow 0$
as
$k\rightarrow\infty$.
This implies that \eqref{eq:w1_in_proof0} holds also in the limit $k\rightarrow\infty$ and
\begin{equation}\label{eq:w1_in_proof_infty}
w_1 = \sum_{j=1}^{\infty} \tfrac{\lambda^j}{(j-1)!} N_j^* \, w_1
= (\be_1^T\otimes I) \Big(\sum_{j=1}^\infty\lambda^j(\bS^T\otimes I)^{j-1}\bN^* w_1\Big).
\end{equation}
In the last equality in \eqref{eq:w1_in_proof_infty} we used that
\begin{equation}\label{eq: Sj}
\bS^j \be_k = \tfrac{(k-1)!}{(j+k-1)!} \, \be_{k+j}
\end{equation}
and therefore
$(\be_k^T\otimes I) (\bS^T\otimes I)^{j-1}=
\be_k^T(\bS^T)^{j-1}\otimes I=\tfrac{(k-1)!}{(j+k-2)!}\be_{k+j-1}^*\otimes I$
for any $k$, as well as $(\be_j^T\otimes I)\bN^*= N_j^*$.
By setting $z=w_1$ we have with \eqref{eq:w1_in_proof_infty} proven the
first block row of \eqref{eq:bw_equiv}. The proof of the other rows follows from induction,
since assuming that $w_k=(\be_k^T\otimes I)\bw$, where $\bw$ is the right-hand
side of \eqref{eq:bw_equiv}, and using \eqref{eq:Nth_row_in_proof}
we find that $w_{k+1}=(\be_{k+1}^*\otimes I)\bw$.

To show (ii), first assume that $\bw\in\ell_1(\rho)$ satisfies
$\lambda \bA\bw=\bw$. This is the same assumption
as in (i) and therefore \eqref{eq:w1_in_proof_infty}
is satisfied. By setting $y=M(0)^{-*}z = M(0)^{-*}w_1$, we have
that $M(0)^*y=\sum_{j=1}^\infty M^{(j)}(0)^*y$, i.e., \eqref{eq:nepb}
is satisfied.

To show the backward implication in (ii), we now assume that $(\lambda,y)$ is a solution to \eqref{eq:nepb}.
Let $z=M(0)^{*}y$ and define a vector $\bw$ as \eqref{eq:bw_equiv}. Then
\begin{align*}
\lambda\bA^*\bw
&= \lambda\displaystyle\sum_{j=1}^{\infty}(\be_1^T\otimes \bN^*
    + \bS^T\otimes I )(\bS^T\otimes I)^{j-1}\bN^*\lambda^j z \\[1mm]
&= \lambda \bN^*\displaystyle\sum_{j=1}^{\infty} \tfrac{1}{(j-1)!} (\be_j^T \otimes I) \bN^*\lambda^{j} z
    + \lambda \displaystyle\sum_{j=1}^{\infty} (\bS^T\otimes I)^{j}\bN^*\lambda^{j} z \\[1mm]
&= \lambda \bN^* \displaystyle\sum_{j=1}^{\infty} \tfrac{-1}{j!} M^{(j)}(0)^* y
    + \displaystyle\sum_{j=2}^{\infty} (\bS^T\otimes I)^{j-1}\bN^* \lambda^{j} z \\[1mm]
&= \bN^* \lambda z + \displaystyle\sum_{j=2}^{\infty} (\bS^T\otimes I)^{j-1}\bN^* \lambda^{j} z
= \displaystyle\sum_{j=1}^{\infty} (\bS^T\otimes I)^{j-1}\bN^*\lambda^{j} z
= \bw.
\end{align*}
To show $\bw\in\ell_1(\rho)$ we now study the weighted $\ell_1$-norm,
\begin{align}
\sum_{k=1}^{\infty}\tfrac{\rho^k}{k!}\|w_k\| &
\le \sum_{k=1}^{\infty}\tfrac{\rho^k}{k!}\sum_{j=1}^\infty
\tfrac{|\lambda|^j(k-1)!}{(j+k-2)!}\|M^{(k+j-1)}(0)^*\|\|\widehat{y}\| \label{eq:wnorm_in_proof} \\
& \le \sum_{k=1}^{\infty}\tfrac{\rho^k}{k!}\sum_{j=1}^\infty
M_\rho\tfrac{|\lambda|^j(k-1)!(k+j-1)!}{(j+k-2)!\rho^{j+k-1}}\|\widehat{y}\|
 = \tfrac{M_\rho\|\widehat{y}\|}{r}
\sum_{j=1}^\infty\sum_{k=1}^\infty\tfrac{|\lambda|^j}{\rho^j}\tfrac{j+k-1}{k!}, \nonumber
\end{align}
where, since $M$ is analytic, there exists a constant $M_\rho$ such that
$\|M^{(j)}(0)\|\le M_\rho\tfrac{j!}{\rho^j}$.
Now note that Taylor expansion of $e^x$ gives the explicit expression
$\sum_{k=1}^\infty \tfrac{j+k-1}{k!}=(j-1)(e-1)+e$.
By combining this with \eqref{eq:wnorm_in_proof} and $|\lambda|<\rho$ we find that the
right-hand side of \eqref{eq:wnorm_in_proof} is
finite and therefore $\bw\in\ell_1(\rho)$.
\end{proof}


\section{Derivation of the infinite bi-Lanczos method}\label{sec:bilanczos}
The algorithm proposed in this paper is based on the Lanczos method for non-Hermitian eigenvalue
problems specified in \cite[Section~7.8.1]{Bai:2000:TEMPLATES}.
We first introduce the standard method
and then adapt the algorithm in such a way that it can be used for
the infinite-dimensional problem.
\subsection{The bi-Lanczos method for standard eigenvalue problems}\label{sec:standardBiLan}
We now briefly summarize the version of the two-sided Lanczos
(also called the bi-Lanczos method) that we use in our derivation.
The method, presented in Algorithm $1$, uses an oblique projection building two bi-orthogonal subspaces
for the simultaneous approximation of left and right eigenvectors.
The short recurrences that are typical for this method lead to far less
storage requirements with respect to orthogonal projection methods for the same problem.
However, as is well known, the method suffers from the loss of bi-orthogonality in finite precision arithmetic.
One can either accept the loss and take more steps, or, if desired, one can re-biorthogonalize all vectors in each iteration. A compromise between these two options is to maintain semiduality as proposed in \cite{Day:1997:Lanczos}. 
For information on various types of breakdowns, how to continue after a breakdown, and how to detect (near) breakdowns, we refer to \cite[Section~7.8.1]{Bai:2000:TEMPLATES}, 
and \cite{Freund:1993:Implementation}.
\begin{table}[htbp!]
{\footnotesize
\noindent \vrule height 0pt depth 0.5pt width \textwidth \\
{\bfseries Algorithm 1:} Bi-Lanczos \\[-3mm]
\vrule height 0pt depth 0.3pt width \textwidth \\
{\bf Input: } Vectors $q_1$, $\til{q}_1$, with $\til{q}_1^* \, q_1 = 1$, $\gamma_1 = \beta_1 = 0$, $q_0 = \til{q}_0 = 0$.\\
{\bf Output: }
Approximate eigentriplets $(\theta_i^{(j)}, x_i^{(j)}, y_i^{(j)})$ of $A$.
\\[-2mm]
\vrule height 0pt depth 0.3pt width \textwidth \\[1mm]
{\bf for }$j=1,2,\ldots,$ until convergence\\
\phantom{M1} (1) \quad $r = Aq_j$\\
\phantom{M1} (2) \quad $s = A^*\til{q}_j$\\
\phantom{M1} (3) \quad $r := r - \gamma_j \, q_{j-1}$\\
\phantom{M1} (4) \quad $s := s - \bar{\beta}_j \, \til{q}_{j-1}$\\
\phantom{M1} (5) \quad $\alpha_j = \til{q}_j^*r$\\
\phantom{M1} (6) \quad $r := r - \alpha_j \, q_j$ \\
\phantom{M1} (7) \quad $s := s - \bar{\alpha}_j \, \til{q}_j$\\
\phantom{M1} (8) \quad $\omega_j = r^*s$\\
\phantom{M1} (9) \quad $\beta_{j+1} = |\omega_j|^{1/2}$ \\
\phantom{M} (10) \quad $\gamma_{j+1} = \bar{\omega}_j / \beta_{j+1}$ \\
\phantom{M} (11) \quad $q_{j+1} = r / \beta_{j+1}$ \\
\phantom{M} (12) \quad $\til{q}_{j+1} = s / \bar{\gamma}_{j+1}$ \\
\phantom{M} (13) \quad Compute eigentriplets $(\theta_i^{(j)}, z_i^{(j)}, \til{z}_i^{(j)})$ of $T_j$. \\
\phantom{M} (14) \quad Test for convergence. \\
\phantom{M} (15) \quad Rebiorthogonalize if necessary. \\
{\bf end}\\
(16) Compute approximate eigenvectors $x_i^{(j)} = Q_jz_i^{(j)}$, $y_i^{(j)} = \til{Q}_j\til{z}_i^{(j)}$.\\[-2mm]
\vrule height 0pt depth 0.3pt width \textwidth
}
\end{table}
After $k$ iterations we obtain the relations:
\begin{align*}
AQ_k            & = Q_kT_k + \beta_{k+1}q_{k+1}e_k^T,\\[1mm]
A^*\!\til{Q}_k  & = \til{Q}_kT^*_k + \bar{\gamma}_{k+1}\til{q}_{k+1}e_k^T,\\[1mm]
\til{Q}_k^*Q_k  &= I_k,
\end{align*}
where for $i=1,\ldots,k$ the columns of $Q_k$ are equal to the vectors $q_i$,
and $\til{q}_i$ are the columns of $\til{Q}_k$,
$e_k$ is the $k$th unit vector,
and the tridiagonal matrix $T_k$ is defined as
\[
T_k = \left[\begin{array}{ccccc}
\alpha_1    & \gamma_2  &         &          \\
\beta_2     & \alpha_2  & \ddots  &          \\
            & \ddots    & \ddots  & \gamma_k \\
            &           & \beta_k & \alpha_k
\end{array}\right].
\]
Furthermore, the relations $\til{q}_{k+1}^*Q_{k}=0$ and $\til{Q}_{k}^*q_{k+1}=0$ hold.
After $k$ iterations, one can compute the eigentriplets $(\theta_i^{(k)}, z_i^{(k)}, \til{z}_i^{(k)})$,
$i = 1,2,\ldots,k$, of $T_k$.
The Ritz values $\theta_i^{(k)}$ are the approximate eigenvalues of $A$,
and the corresponding right and left Ritz vectors are
$x_i^{(k)} = Q_kz_i^{(k)}$ and $y_i^{(k)} = \til{Q}_k\til{z}_i^{(k)}$, respectively.

\subsection{Krylov subspace and infinite-dimensional vector representations}
For our infinite-dimensional problem where we work with the matrix $\bA$ and vectors of infinite length, we need to build infinite-dimensional Krylov spaces,
$\mathcal{K}_k(\bA,\bx)$ and $\mathcal{K}_k(\bA^*,\til{\by})$, for some starting vectors $\bx$ and $\til{\by}$ of infinite length.
To adapt the algorithm to the infinite-dimensional problem we have to address the issue of storing vectors with infinite length.
By choosing the starting vectors carefully we will be able to store only a finite number of vectors of
length $n$.
The Krylov subspaces will contain vectors that are consistent with eigenvector approximations.
\begin{proposition}\label{prop:krylovstruc}
Suppose $\bx=\be_1\otimes x_1$ and $\til{\by}=\bN^* \til{y}_1$, where $x_1, \til{y}_1\in\CC^n$.
\begin{enumerate}
\item[(a)]
For any $k\in\NN$, $\displaystyle \bA^k \bx=\sum_{j=1}^{k+1} (\be_j\otimes z_{k-j+1})$,
where $z_0 = \tfrac{1}{k!}x_1$ and for \\ $i~\in~\{1,\ldots, k\}$ $z_i$ is given by the recursion
$\displaystyle z_i = \displaystyle\sum_{\ell=1}^i \tfrac{(k-i+\ell)!}{(\ell-1)!(k-i)!}N_{\ell} z_{i-\ell}$.
\item[(b)] For any $k\in\NN$, $\displaystyle(\bA^*)^k\til{\by}=\sum_{j=1}^{k+1}(\bS^T \otimes I)^{j-1}\bN^* \til{z}_{k-j+1}$,
where $\til{z}_0 = \til{y}_1$ and for $i\in\{1,\ldots, k\}$ $\til{z}_i$ is given by the recurrence relation
$\displaystyle \til{z}_i = \sum_{\ell=1}^i \tfrac{1}{(\ell-1)!}N^*_{\ell}\til{z}_{i-\ell}$.
\end{enumerate}
\end{proposition}
\begin{proof}
\begin{enumerate}
\item[(a)]
It is easily seen that the result holds for $k=1$, when $z_0 = x_1$ and $z_1= N_1z_{0}$.
Suppose the result holds for $k-1$,
thus \[\bA^{k-1}\bx = \sum_{j=1}^{k} (\be_j \otimes a_{k-j}),\]
where $a_0 = \tfrac{1}{(k-1)!}x_1$ and $a_i = \displaystyle\sum_{\ell=1}^{i} \tfrac{(k-i+\ell-1)!}{(\ell-1)!(k-i-1)!}N_{\ell}a_{i-\ell}$ for $i\in\{1,\ldots,k-1\}$.
Then, using \eqref{eq: Sj},
\begin{align*}
\bA^k\bx
&= \displaystyle\sum_{j=1}^k (\be_1\otimes \bN)(\be_j\otimes a_{k-j})
    + \displaystyle\sum_{j=1}^k (\bS \otimes I) (\be_j\otimes a_{k-j}) \\
&= \displaystyle\sum_{j=1}^k (\be_1\otimes N_ja_{k-j})
    + \displaystyle\sum_{j=1}^k \tfrac{(j-1)!}{j!}(\be_{j+1}\otimes a_{k-j}) \\
&= \displaystyle (\be_1\otimes \sum_{j=1}^k N_ja_{k-j})
    + \displaystyle\sum_{j=2}^{k+1} (\be_{j}\otimes \tfrac{1}{(j-1)} a_{k-j+1}).
\end{align*}
Defining $z_k = \displaystyle\sum_{j=1}^k N_ja_{k-j}$ and $z_{k-j+1} = \tfrac{1}{(j-1)} a_{k-j+1}$,
it can be seen that all $z_i$ are as stated in $(a)$. This shows (a) by induction.

\item[(b)]
It is easily seen that for $k=1$ the result holds, where $\til{z}_0 = \til{y}_1$ and  $\til{z}_1= N_1\til{z}_{0}$.
Suppose the proposition holds for $k-1$. Then
\begin{align*}
(\bA^*)^k\til{\by}
&= \displaystyle\sum_{j=1}^k (\be_1^T \otimes \bN^*)(\bS^T \otimes I)^{j-1} \bN^* z_{k-j}
    + \sum_{j=1}^k(\bS^T \otimes I)^{j}\bN^*\til{z}_{k-j} \\
&= \displaystyle\sum_{j=1}^k \tfrac{1}{(j-1)!}(\be_j^T \otimes \bN^*) \bN^* \til{z}_{k-j}
    + \sum_{j=2}^{k+1} (\bS^T \otimes I)^{j-1}\bN^*\til{z}_{k-j+1} \\
&= \displaystyle \bN^* \sum_{j=1}^{k} \tfrac{1}{(j-1)!} N_j^* z_{k-j}
    + \sum_{j=2}^{k+1}(\bS^T \otimes I)^{j-1}\bN^* \til{z}_{k-j+1} \\
&= \sum_{j=1}^{k+1}(\bS^T \otimes I)^{j-1}\bN^* \til{z}_{k-j+1},
\end{align*}
where $\til{z}_0 = \til{y}_1$, $\til{z}_1$, \dots, $\til{z}_m$ are
as stated under (b). This proves (b) by induction.
\end{enumerate}
\end{proof}
As we have seen in Theorems \ref{thm:righteigvec} and \ref{thm:lefteigvec},
the right and left eigenvectors of interest have the form \eqref{eq:bv_equiv} and \eqref{eq:bw_equiv}, respectively.
Proposition~\ref{prop:krylovstruc} has shown that by choosing starting vectors $\bx=\be_1\otimes x_1$
and $\til{\by}=\bN^* \til{y}_1$ the vectors that span the Krylov subspaces are of the form
\begin{subequations}\label{eq:inf vectors}
\begin{align}
  \ba       &= \sum_{j=1}^{k_a} (\be_j\otimes a_j),       \label{eq:ba} \\
  \til{\ba} &= \sum_{j=1}^{k_{\til{a}}}(\bS^T \otimes I)^{j-1}\bN^* \til{a}_j,   \label{eq:tba}
\end{align}
\end{subequations}
respectively.
Also linear combinations of vectors from the same Krylov subspaces,
and therefore also the approximate eigenvectors, will be of this form.
We will distinguish the two types of vectors \eqref{eq:ba} and \eqref{eq:tba} by a tilde.
These vectors can be seen as a truncated version of the vectors in \eqref{eq:bv_equiv} and \eqref{eq:bw_equiv}.
Vectors of the form \eqref{eq:ba} have a finite number of nonzeros and therefore storing only
the nonzero entries gives a finite representation of the vector of infinite length, i.e.,
by storing the vectors  $a_j$, for $j = 1,\ldots, k_{a}$.
The vectors of infinite length of type \eqref{eq:tba} can also be stored with a finite number
of vectors in $\mathbb{C}^n$, namely by storing the vectors  $\til{a}_j$, for $j = 1,\ldots, k_{\til{a}}$.

\subsection{Scalar products and matrix-vector products}
The previously introduced types of infinite-dimensional vectors, \eqref{eq:ba} and \eqref{eq:tba},
will be used in the algorithm for the infinite-dimensional problem.
Various operations involving these types of vectors of infinite length,
such as scalar products and matrix-vector products,
have to be adapted to the infinite-dimensional case.
First we introduce two different scalar products.
\begin{lemma} Suppose $\ba,\bb\in\CC^\infty$ are two vectors of type
$\eqref{eq:ba}$ given by $\displaystyle \ba=\sum_{j=1}^{k_a} (\be_j\otimes a_j)$
and $\displaystyle \bb=\sum_{j=1}^{k_b} (\be_j\otimes b_j)$.
Then, 
\begin{equation}
\ba^*\bb=\displaystyle\sum_{j=1}^{\min(k_a,k_b)} a_j^*b_j.
\end{equation}
\end{lemma}
\begin{proof}
This follows straightforwardly from the definition of the vectors.
\end{proof}

Another scalar product used in the bi-Lanczos algorithm
is a product of vectors of type \eqref{eq:ba} and \eqref{eq:tba}.
It can be computed efficiently in infinite dimensions as explained in the next proposition.
\begin{theorem} Suppose $\til{\ba},\bb\in\CC^\infty$ are of type
\eqref{eq:tba} and \eqref{eq:ba}, respectively, given by
$\displaystyle \til{\ba} = \sum_{j=1}^{k_{\til{a}}} (\bS^T \otimes I)^{j-1}\bN^*\til{a}_j$
and $\displaystyle \bb = \sum_{\ell=1}^{k_b} (\be_{\ell}\otimes b_\ell)$.
Then,
\begin{equation}\label{eq:leftright}
\til{\ba}^*\bb=\sum_{j=1}^{k_{\til{a}}}\sum_{\ell=1}^{k_b}
\tfrac{(\ell-1)!}{(j+\ell-2)!}\til{a}_j^*N_{j+\ell-1} b_\ell.
\end{equation}
\end{theorem}\label{prop:leftright_ip}
\begin{proof}
This can be derived directly via the following equality
\[ \til{\ba}^*\bb
= \displaystyle\sum_{j=1}^{k_{\til{a}}} \sum_{\ell=1}^{k_b} \Big((\bS^T \otimes I)^{j-1}\bN^*\til{a}_j\Big)^*
                          (\be_{\ell}\otimes b_{\ell})
= \displaystyle\sum_{j=1}^{k_{\til{a}}} \sum_{\ell=1}^{k_b} \til{a}_j^*\bN
                          \Big(\be_{j+\ell-1} \otimes \tfrac{(\ell-1)!}{(j+\ell-2)!} b_{\ell}\Big).
\]
\end{proof}

To translate the finite dimensional matrix-vector multiplication to the infinite-dimensional case
two variants of matrix-vector products have to be investigated,
one with the matrix $\bA$ and a vector of type \eqref{eq:ba},
and one with the matrix $\bA^*$ and a vector of type \eqref{eq:tba}.
\begin{theorem}[Action of $\bA$] 
\label{thm:Aaction}
Suppose $\ba\in\CC^\infty$ is of type \eqref{eq:ba} given by
$\displaystyle \ba=\sum_{j=1}^{k_a} (\be_j\otimes a_j)$.
Then,
\vspace{-2mm}
\begin{equation}\label{eq:Aaction}
 \bA\ba=\sum_{j=1}^{k_a+1}(\be_j\otimes b_j),
\end{equation}
where 
\begin{equation}\label{eq:Aactionb}
b_j= \tfrac{1}{j-1}a_{j-1} \textrm{ for }j=2,\ldots,k_a+1\textrm{, and }
b_1= \displaystyle\sum_{j=1}^{k_a} N_j a_j.
\end{equation}
\end{theorem}
\begin{proof}
This can be proven with induction. The computation is analogous to the one needed in the proof of Proposition \ref{prop:krylovstruc}$(a)$.
\end{proof}

\begin{theorem}[Action of $\bA^*$]
\label{thm:Ataction}
Suppose $\til{\ba}\in\CC^\infty$ is of type \eqref{eq:tba} given by
$\displaystyle \til{\ba}=\sum_{j=1}^{k_{\til{a}}} (\bS^T \otimes I)^{j-1}\bN^*\til{a}_j$.
Then,
\vspace{-2mm}
\begin{equation}\label{eq:Ataction}
  \bA^*\til{\ba}=\sum_{j=1}^{k_{\til{a}}+1} (\bS^T \otimes I)^{j-1}\bN^*\til{b}_j,
\end{equation}
where 
  \begin{equation} \label{eq:Atactionb}
\til{b}_j = \til{a}_{j-1}\textrm{ for }j=2,\ldots k_{\til{a}}+1\textrm{, and }
\til{b}_1 = \displaystyle\sum_{j=1}^{k_{\til{a}}}\tfrac{1}{(j-1)!}N_j^*\til{a}_j.    
  \end{equation}
\end{theorem}
\begin{proof}
Analogous to the computation in the proof of Proposition \ref{prop:krylovstruc}$(b)$.
\end{proof}

\subsection{The infinite bi-Lanczos method}\label{sec:infbilanczos}
In Algorithm 2 we present the bi-Lanczos algorithm for the infinite-dimensional problem. 
It is set up analogously to the algorithm of the standard bi-Lanczos method: 
all line numbers are corresponding. 
As we have seen, every vector of infinite length can be represented 
by a finite number of vectors of length $n$.
In the algorithm these vectors of infinite length are denoted by matrices 
whose columns correspond to the length-$n$ vectors representing 
the infinite-dimensional vector.
The index of the matrices in the new algorithm indicate 
the number of columns of the matrix,
i.e., $R_k\in\CC^{n\times k}$, and we denote the $\ell$th column of $R_k$ by $R_{k,\ell}$,
i.e., $R_k=[R_{k,1},\ldots,R_{k,k}]$.

We now describe certain steps of the algorithm illustrating that the
algorithm can be implemented with matrices of finite size. This first
description can be considerable improved by reducing the number of necessary
linear solves as we shall explain in Section~\ref{sec:representation}.
\begin{enumerate}
\item[(1)+(2)] In the first two lines two matrix-vector multiplications are executed, the first with the infinite-dimensional matrix $\bA$ and a vector of type \eqref{eq:ba}, the second with the infinite-dimensional matrix $\bA^*$ and a vector of type \eqref{eq:tba}. In Theorems \ref{thm:Aaction} and \ref{thm:Ataction} it is shown how these actions are performed. A vector represented by $k$ vectors of length $n$ results after a multiplication with $\bA$ or $\bA^*$ in a vector represented by $k+1$ vectors of length $n$.
More precisely, the action \eqref{eq:Aaction} and \eqref{eq:Ataction} can be computed
with the formulas
\begin{subequations}\label{eq:Aaction_comp}
  \begin{eqnarray}  
   b_j&=& \tfrac{1}{j-1}a_{j-1} \textrm{ for }j=2,\ldots,k_a+1\\
   b_1&=& -M(0)^{-1}\displaystyle\sum_{j=1}^{k_a} \frac{1}{j}M^{(j)}(0)a_j,
\end{eqnarray}
\end{subequations}
and
\begin{subequations}\label{eq:Ataction_comp}
  \begin{eqnarray}  
\til{b}_j &=& \til{a}_{j-1}\textrm{ for }j=2,\ldots k_{\til{a}}+1\\
\til{b}_1 &=& -\displaystyle\sum_{j=1}^{k_{\til{a}}}M^{(j)}(0)^*\left(\tfrac{1}{j!}M(0)^{-*}\til{a}_j\right).    \label{eq:Ataction_comp_b}
  \end{eqnarray}
\end{subequations}
At first sight \eqref{eq:Ataction_comp} appears  
to require a ridiculous  number of linear solves. We show 
 in Section~\ref{sec:representation} how these linear solves can be avoided by using
a particular implicit representation of $\til{a}_j$.
\item[(3)+(4)] The new vectors are orthogonalized against a previous vector. Linear combinations of a vector of the form \eqref{eq:ba} (or \eqref{eq:tba}) are again of that form, and thus can be represented as such. The new vectors are represented by $k+1$ vectors of length $n$, while the previous vectors are represented by $k-1$ length-$n$ vectors. To enable the summation we add two zero columns to the $n\times (k-1)$-matrices representing the previous vectors. 
\item[(5)+(8)] The coefficients computed in this step are needed for the orthogonalization of the vectors, and furthermore they are the entries of the tridiagonal matrix $T_k$. The computation of these coefficients involves an inner product between a vector of type \eqref{eq:ba} and one of type \eqref{eq:tba}, and is executed as described in Proposition~\ref{prop:leftright_ip}. 
More specificially the theory in Theorem~\ref{prop:leftright_ip} is in our setting specialized to the explicit
formulas
\begin{subequations}\label{eq:leftright_comp}
\begin{eqnarray} 
   \til{\ba}^*\bb&=&-\sum_{j=1}^{k_{\til{a}}} \til{a}_j^* M(0)^{-1} \left(\sum_{\ell=1}^{k_b}  M^{(j+\ell-1)}(0) \tfrac{(\ell-1)!}{(j+\ell-1)!}b_\ell\right)\\
&=&-\sum_{j=1}^{k_{\til{a}}} (M(0)^{-*}\til{a}_j)^*  \left(\sum_{\ell=1}^{k_b}  M^{(j+\ell-1)}(0) \tfrac{(\ell-1)!}{(j+\ell-1)!}b_\ell\right) \label{eq:leftright_comp_b}
\end{eqnarray}
\end{subequations}
Similar to formula \eqref{eq:Ataction_comp}, we show how to reduce the number of linear solves
in Section~\ref{sec:representation}.
\item[(6)+(7)] These orthogonalization steps are comparable to those in $(3)+(4)$. Since the vectors are orthogonalized against the previous vector, one column of zeros is added to allow for the summation.
\item[(15)] An important property of this new method is that the computation of the approximate eigenvectors for the solution of \eqref{eq:nep} entails the storage of (only) $k$ vectors of length $n$ for each subspace.
To clarify this, recall from Section~\ref{sec:standardBiLan} that from
the eigentriplet $(\theta_1^{(k)}, z_1^{(k)}, \til{z}_1^{(k)})$ of $T_k$
we can deduce an approximate eigentriplet $
(\theta_1^{(k)}, Q_kz_1^{(k)}, \til{Q}_k\til{z}_1^{(k)})$ for $\bA$.
The approximate right eigenpair $(\theta_1^{(k)}, Q_kz_1^{(k)})$ approximates thus a right eigenpair of $\bA$ that has the form $(\lambda, \bv)$, where $\bv = \left[\tfrac{\lambda^{j-1}}{(j-1)!} \, x\right]_{j=1}^\infty$ (see \eqref{eq:bv_equiv}). 
From this approximate pair of $\bA$, we are able to extract an approximate solution to $\eqref{eq:nepa}$. 
Note that the columns of $Q_k$ represent vectors of type \eqref{eq:ba}
and thus a linear combination of the columns is itself
a representation of a vector of this type.
Suppose $s_r$ stands for the first $n$-length block of $Q_kz_1^{(k)}$. Then $s_r$ is an approximation to $x$, the first length-$n$ block of $\bv$, and thus, by Theorem~\ref{thm:righteigvec} $((\theta_1^{(k)})^{-1}, s_r)$ is an approximate solution to $\eqref{eq:nepa}$.

Similarly, the left eigenpair $(\theta_1^{(k)}, \til{Q}_k\til{z}_1^{(k)})$ approximates a left eigenpair of $\bA$ that has the form $(\lambda, \bw)$, where $\bw=\sum_{j=1}^\infty (\bS^T\otimes I)^{j-1}\bN^*\lambda^j z$ (see \eqref{eq:bw_equiv}). 
Again, we can deduce an approximate solution to \eqref{eq:nepb} from this approximate pair of $\bA$. 
The columns of $\til{Q}_k$ represent vectors of type \eqref{eq:tba} such that a linear combination of the columns is itself
a representation of a vector of this type.
Suppose the first $n$-length block of $\til{Q}_k\til{z}_1^{(k)}$ is  called $s_{\ell}$. By Theorem~\ref{thm:lefteigvec} we know that $s_{\ell}$ is an approximation to $\lambda z = \lambda M(0)^*y$. Hence $((\theta_1^{(k)})^{-1}, \theta_1^{(k)}M(0)^{-*}s_{\ell})$ is an approximate solution to \eqref{eq:nepb}.

To recover the approximate eigenvectors of \eqref{eq:nep} we do not have to store the entire matrices $Q_k$ and $\til{Q}_k$. As we have just shown, the storage of (only) $k$ vectors of length $n$ for each subspace is sufficient. 
\end{enumerate}
\begin{table}[htbp!]
{\footnotesize
\noindent \vrule height 0pt depth 0.5pt width \textwidth \\
{\bfseries Algorithm 2:} Infinite bi-Lanczos \\[-3mm]
\vrule height 0pt depth 0.3pt width \textwidth \\
{\bf Input: } Vectors $q_1,\til{q}_1\in\CC^n$, with $\til{q}_1^*\!M'(0)q_1 = 1$, $P_0 = \til{P}_0 = [ \phantom{0}]$, $P_1 = [q_1]$, $\til{P}_1 = [\til{q}_1]$, $\gamma_1 = \beta_1 = 0$. \\
{\bf Output: } Approximate eigentriplets $((\theta_i^{(k)})^{-1}, x_i^{(k)}, y_i^{(k)})$ to nonlinear eigenvalue problem $(\ref{eq:nep})$.
\\[-3mm]
\vrule height 0pt depth 0.3pt width \textwidth \\[1mm]
{\bf for }$k=1,2, \ldots,$ until convergence\\
\phantom{M1} (1) \quad Compute $R_{k+1}:=[b_1,\ldots,b_{k+1}]\in\CC^{n\times (k+1)}$ with \eqref{eq:Aaction_comp} where, $k_a=k$,   \\
\phantom{M1  (1)} \quad $a_\ell=P_{k,\ell}$ for $\ell=1,\ldots,k$.\\
\phantom{M1} (2) \quad Compute $\til{R}_{k+1}:=[\til{b}_1,\ldots,\til{b}_{k+1}]\in\CC^{n\times (k+1)}$ with \eqref{eq:Ataction_comp} where, $k_{\til{a}}=k$,   \\
\phantom{M1  (2)} \quad $\til{a}_\ell=\til{P}_{k,\ell}$ for $\ell=1,\ldots,k$.\\
\phantom{M1} (3) \quad $R_{k+1} = R_{k+1} - \gamma_k[P_{k-1},0,0]$\\
\phantom{M1} (4) \quad $\til{R}_{k+1} = \til{R}_{k+1} - \bar{\beta}_{k}[\til{P}_{k-1},0,0]$ \\
\phantom{M1} (5) \quad Compute $\alpha_k=\til{\ba}^*\bb$ with \eqref{eq:leftright_comp}
where
$\til{a}_\ell=\til{P}_{k,\ell}$, $\ell=1,\ldots,k$, and \\
\phantom{M1  (5)} \quad  $b_\ell=R_{k+1,\ell}$ for $\ell=1,\ldots,k+1$ and $k_{\til{a}}=k$ and $k_b=k+1$.\\
\phantom{M1} (6) \quad $R_{k+1} = R_{k+1} - \alpha_k[P_{k},0]$\\
\phantom{M1} (7) \quad $\til{R}_{k+1} = \til{R}_{k+1} - \bar{\alpha}_{k}[\til{P}_{k},0]$ \\
\phantom{M1} (8) \quad Compute $\omega_k = \bar{\til{\ba}^*\bb}$ with \eqref{eq:leftright_comp}
where $\til{a}_\ell=\til{R}_{k+1,\ell}$, $b_\ell=R_{k+1,\ell}$ \\
\phantom{M1 (8)} \quad for $\ell=1,\ldots,k+1$, where $k_{\til{a}}=k_b=k+1$.\\
\phantom{M1} (9) \quad $\beta_{k+1} = |\omega_k|^{1/2}$\\
\phantom{M} (10) \quad $\gamma_{k+1} = \bar{\omega}_k / \beta_{k+1}$\\
\phantom{M} (11) \quad $P_{k+1} = R_{k+1} / \beta_{k+1}$\\
\phantom{M} (12) \quad $\til{P}_{k+1} = \til{R}_{k+1} / \bar{\gamma}_{k+1}$\\
\phantom{M} (13) \quad Compute eigentriplets $(\theta_i^{(k)}, z_i^{(k)}, \til{z}_i^{(k)})$ of $T_k$. \\
\phantom{M} (14) \quad Test for convergence. \\
{\bf end}\\
(15) Compute approximate eigenvectors $x_i^{(k)}$ and $y_i^{(k)}$.\\[-2mm]
\vrule height 0pt depth 0.5pt width \textwidth
}
\end{table}

\subsection{Computational representation of the infinite vectors} \label{sec:representation}
The algorithm described in the previous subsection is complete in the sense
that it shows how one can carry out the two-sided Lanczos for the infinite matrix $\bA$ in
finite-dimensional arithmetic. However, it needs
several modifications to become a practical algorithm. Most importantly, 
by inspection of  \eqref{eq:Ataction_comp} and
\eqref{eq:leftright_comp} we directly conclude that it 
requires  a large number of linear solves corresponding to $M(0)^{-1}$ and $M(0)^{-*}$. 
With a change of variables we  now show how the number of linear solves per step can be reduced to two linear solves
per iteration by a particular representation of $\til{\ba}$ and $\til{\bb}$.

The choice of representation is motivated by 
 the fact that the computational formulas involving the vectors $\til{a}_j$ appear 
in combination with a linear solve with $M(0)^{-*}$, in particular in formulas \eqref{eq:Ataction_comp_b} and \eqref{eq:leftright_comp_b}. 
This property is also naturally expected from the fact that any infinite vector of type \eqref{eq:tba} can 
be factorized as
\begin{align*}
\til{\ba} &= \displaystyle\sum_{j=1}^{k_{\til{a}}}(\bS^T \otimes I)^{j-1}\bN^* \til{a}_j \\
		  &= \displaystyle - \sum_{j=1}^{k_{\til{a}}}(\bS^T \otimes I)^{j-1}\
		  \begin{bmatrix}
		  M'(0)\! & \!\tfrac{1}{2}M^{(2)}(0)\! & \!\ldots \end{bmatrix}^*M(0)^{-*} \til{a}_j.
\end{align*}
Instead of storing the $k_{\til{a}}$ vectors $\til{a}_j$ that represent the infinite vector $\til{\ba}$,
and storing the $k_{\til{b}}$ vectors $\til{b}_j$ that represent the infinite vector $\til{\bb}$,
 we store the vectors  
\begin{subequations}\label{eq:tilde_comp}
\begin{eqnarray}
\til{a}_j^{\comp}&:=&M(0)^{-*} \til{a}_j,\textrm{ for }j=1,\ldots,k_{\til{a}}\\
\til{b}_j^{\comp}&:=&M(0)^{-*} \til{b}_j,\textrm{ for }j=1,\ldots,k_{\til{b}}.
\end{eqnarray}
\end{subequations}
The superscript \emph{\comp} is used to indicate that this vector is the representation which is used
in the computation.

Some additional efficiency can be achieved by also modifying the representation
of $\ba$ and $\bb$. Instead of representing these vectors with $a_j$, $j=1,\ldots,k_a$ 
and $b_j$, $j=1,\ldots,k_b$, we set
\begin{subequations}\label{eq:tilde2_comp}
\begin{eqnarray}
a_j^{\comp}&:=&(j-1)! a_j,\textrm{ for }j=1,\ldots,k_{\til{a}}\\
b_j^{\comp}&:=&(j-1)! b_j,\textrm{ for }j=1,\ldots,k_{\til{b}}.
\end{eqnarray}
\end{subequations}
This reduces the number of scalar operations and simplifies the implementation.

The substitutions \eqref{eq:tilde_comp} and \eqref{eq:tilde2_comp} 
translate
the steps of the algorithm as follows.
Since the substitution is linear and changes both the representation
of $\til{\ba}$ and of $\ba$, the operations associated with Step~(1), (2), (5) and (8) need
to be modified.


The substitution \eqref{eq:tilde2_comp} changes the operations
associated with the action of $\bA$ in Step (1). Instead of \eqref{eq:Aaction_comp}
we use
\begin{subequations}\label{eq:Aaction_comp2}
  \begin{eqnarray}  
   b_j^{\comp}&=& a_{j-1}^{\comp} \textrm{ for }j=2,\ldots,k_a+1\label{eq:Aaction_comp2_a}\\
   b_1^{\comp}&=& -M(0)^{-1}\displaystyle\sum_{j=1}^{k_a} M^{(j)}(0)\tfrac{1}{j!}a_j^{\comp}.\label{eq:Aaction_comp2_b}
\end{eqnarray}
\end{subequations}
The reason for this substitution is that \eqref{eq:Aaction_comp2_a} can now
 be computed without any operations on the vectors,
and that \eqref{eq:Aaction_comp2_b} is completely analogous to \eqref{eq:Ataction_comp2_b} with a complex conjugate transpose.

We need to compute the action of $\bA^*$ by using \eqref{eq:Ataction_comp} in Step~(2).
The substitution corresponding to the
representation \eqref{eq:tilde_comp} into 
 \eqref{eq:Ataction_comp} leads to the formulas 
\begin{subequations}\label{eq:Ataction_comp2}
  \begin{eqnarray}  
\til{b}_j^{\comp} &=& \til{a}_{j-1}^{\comp}\textrm{ for }j=2,\ldots k_{\til{a}}+1\\
  \til{b}_1^{\comp} &=& M(0)^{-*}\til{b}_1=
-M(0)^{-*}
\displaystyle\sum_{j=1}^{k_{\til{a}}}M^{(j)}(0)^*\tfrac{1}{j!}\til{a}_j^{\comp} \label{eq:Ataction_comp2_b}
  \end{eqnarray}
\end{subequations}
Note that in contrast to \eqref{eq:Ataction_comp}, \eqref{eq:Ataction_comp2} only involves one linear solve.

We need to compute the scalar product of infinite vectors 
in Step~(5) and (8). Instead of using \eqref{eq:leftright_comp},
we can now reformulate formula \eqref{eq:leftright_comp} with the new representation as 
\begin{eqnarray} 
   \til{\ba}^*\bb&=&
-\sum_{j=1}^{k_{\til{a}}} (M(0)^{-*}\til{a}_j)^*  \left(\sum_{\ell=1}^{k_b}  M^{(j+\ell-1)}(0) \tfrac{(\ell-1)!}{(j+\ell-1)!}b_\ell\right)\notag\\
&=&
-\sum_{j=1}^{k_{\til{a}}} (\til{a}_j^{\comp})^*  \sum_{\ell=1}^{k_b}  M^{(j+\ell-1)}(0) \tfrac{1}{(j+\ell-1)!}b_\ell^{\comp}
 \label{eq:leftright_comp2}
\end{eqnarray}
This formula does not require any linear solve, which should be seen
in contrast to \eqref{eq:leftright_comp} which requires $k_{\til{a}}$ linear solves.
Despite this improvement, we will see in the following section and numerical examples
that the computation of
the scalar product in \eqref{eq:leftright_comp2} is often the dominating part of the algorithm.

\subsection{Complexity considerations and implementation}\label{sec:complexity}


The computational resources and problem specific aspects for the algorithm can be summarized as follows. 
The description below is based on the representation in Section~\ref{sec:representation}.
Again our discussion is conducted with references to the steps of the algorithm. We neglect the computation associated with the scalars in step (9) and (10).
\begin{enumerate}
\item[(1)+(2)] In the representation of Section~\ref{sec:representation} 
we need to evaluate \eqref{eq:Ataction_comp2} and \eqref{eq:Aaction_comp2}. 
The main computational effort of evaluating these formulas consists of 
first computing a linear combination of derivatives, in the sense that we need to
call the functions 
    \begin{subequations}\label{eq:lincombs}
    \begin{eqnarray}
      \operatorname{lincomp}(z_1,\ldots,z_m)&=&\sum_{i=1}^m M^{(i)}(0)z_i   \label{eq:lincomb}\\
      \operatorname{lincombstar}(\til{z}_1,\ldots,\til{z}_m) &=&\sum_{i=1}^m M^{(i)}(0)^*\til{z}_i. \label{eq:lincombstar}
    \end{eqnarray}
    \end{subequations}
   The output of the functions $\operatorname{lincomp}(\cdot)$ and $\operatorname{lincombstar}(\cdot)$
are used for a linear solve associated with $M(0)$ and $M(0)^*$. Note that $M(0)$ and $M(0)^*$ are
not changed throughout the iteration such that for many large and sparse eigenvalue problems 
efficiency improvements can be achieved by computing an LU-factorization before the iteration starts.
\item[(3)+(4)] These steps consist of simple operations on a full matrix of size $n\times k$ and are in general not computationally demanding. The same holds for steps (6)+(7) and (11)+(12) of the algorithm. 
\item[(8)+(9)] The scalar products are computed with \eqref{eq:leftright_comp2}. Note that one part of that formula is a linear combination of derivatives, such that it can be computed by 
calling the function defined in \eqref{eq:lincomb} $k_{\til{a}}$ times, i.e., $\operatorname{lincomb}(\cdot)$. Since, both $k_a$ and $k_{\til{a}}$ increase in every iteration, the double sum in \eqref{eq:leftright_comp2}
accumulates after $k$ steps to a total complexity 
\begin{equation}\label{eq:tscalarprod}
t_{\rm scalarprod}(k,n)=\mathcal{O}(k^3n).
\end{equation}
  \item[(13)] This step consists of computing an eigentriplet of a $k\times k$ tridiagonal matrix, which in general is not a computationally dominating part of the algorithm. 
\end{enumerate}

We conclude that in order to apply our algorithm to a specific problem the
user needs to provide a function to solve linear systems corresponding to $M(0)$ and 
$M(0)^*$ and a procedure to compute linear combinations of derivatives
as defined in  \eqref{eq:lincombs}. This can be seen in relation to IAR \cite{Jarlebring:2012:INFARNOLDI} 
and TIAR \cite{Jarlebring:2015:WTIARTR} where 
the user needs to provide a function to carry out linear solves corresponding to $M(0)$ and
compute linear combinations as in \eqref{eq:lincomb}.

%
%
%
%
%

\begin{remark}[Scalar product complexity and improvement]\label{rem:doublesum}\rm
Both IAR and TIAR have a complexity (in terms of number of floating point operations) of $\mathcal{O}(k^3n)$, although TIAR is in general considerably faster in practice. 
Due to the scalar product complexity  \eqref{eq:tscalarprod} our algorithm also has
a computational complexity $\mathcal{O}(k^3n)$. However, it turns out that
the scalar product computation can be improved in a problem specific case. 

To ease the notation let us collect the vectors 
$\tilde{a}_j^{\comp}$ for $j=1,\ldots,k_{\til{a}}$ in $\tilde{A}\in \RR^{n\times k_{\til{a}}}$ and 
$b_\ell^{\comp}$ for $\ell=1,\ldots,k_{b}$ in $B\in \RR^{n\times k_{b}}$ (which is how the vectors
are stored in Algorithm~2). 
Moreover, without loss of generality we decompose the NEP as
a sum of products of matrices and scalar functions
\[
  M(\lambda)=M_1f_1(\lambda)+\cdots+M_pf_p(\lambda),
\]
where $f_1,\ldots,f_p$ are analytic functions.
Although the assumption is not a restriction of generality, the following
approach is only efficient if $p$ is small. This is the case for many
NEPs, e.g., those in Section~\ref{sec:numexpes}.
The scalar product \eqref{eq:leftright_comp2} is
now
\begin{eqnarray}
   \til{\ba}^*\bb&=&
-\sum_{j=1}^{k_{\til{a}}}\sum_{\ell=1}^{k_b}  \sum_{k=1}^p (\til{a}_j^{\comp})^*   M_kf_k^{(j+\ell-1)}(0) \tfrac{1}{(j+\ell-1)!}b_\ell^{\comp}\notag\\
&=&-\sum_{j=1}^{k_{\til{a}}}\sum_{\ell=1}^{k_b}  \sum_{k=1}^p (\til{a}_j^{\comp})^*   M_kb_\ell^{\comp} \tfrac{1}{(j+\ell-1)!}f_k^{(j+\ell-1)}(0)\notag\\
&=&-\sum_{j=1}^{k_{\til{a}}}\sum_{\ell=1}^{k_b}  \sum_{k=1}^p \hat{M}_{k,j,\ell} \tfrac{1}{(j+\ell-1)!}f_k^{(j+\ell-1)}(0)\label{eq:leftright_comp3}
\end{eqnarray}
where 
\begin{equation}\label{eq:Mhat}
  \hat{M}_k:=AM_kB,\textrm{ for }k=1,\ldots,p.
\end{equation}
The matrices $\hat{M}_1,\ldots,\hat{M}_p$ can be computed before computing the sum.
In this fashion the last line of \eqref{eq:leftright_comp3} is independent of the size 
of the problem $n$. Moreover, the sum in \eqref{eq:leftright_comp3}
can be carried out by appropriate matrix vector products.
This reformulation of the step changes the accumulated computation
time complexity of the scalar product to
\[
  \tilde{t}_{\rm scalarprod}=\mathcal{O}(pk^3)+\mathcal{O}(npk^2),
\]
under the assumption that $M_kB$ is carried out in $\mathcal{O}(nk_b)$. 
In comparison to \eqref{eq:tscalarprod}, this approach is advantageous if $p$ is small and $n$ is large. 
Moreover, in practice the advantage of this appears quite large since
on modern computer architectures matrix-matrix
products are more efficient than (unoptimized) double sums, due to more efficient usage of
CPU-cache. 

\end{remark}
\section{Numerical experiments}\label{sec:numexpes}
Our approach is intended for large and sparse problems, and we illustrate
problem properties by solving two nonlinear problems. The simulations
were carried out with an implementation in MATLAB, 
and using a computer with an Intel Core i5-3360M processor and 8 GB of RAM.
In order to increase reproducability 
of our results we have made the MATLAB-codes freely available online, 
in a way that completely regenerates 
the figures in the following simulations.\footnote{The MATLAB codes are available: \url{http://www.math.kth.se/~eliasj/src/infbilanczos/}}

\subsection{A second order delay-differential equation}
We start  with the illustration of the
properties and competitiveness of the algorithm by computing solutions to an artificial large-scale NEP
stemming from a second order delay-differential equation,
\begin{equation}\label{eq:example1}
M(\lambda)=-\lambda^2I+A_0+e^{-\lambda}A_1,
\end{equation}
where $A_0$ and $A_1$ are randomly generated sparse matrices with
normally distributed random entries.
Solutions to \eqref{eq:example1} can for instance be used 
to study stability of time-delay system. See \cite{Michiels:2007:STABILITYBOOK} for further literature on time-delay systems.
For the experiments we choose the matrices to be of dimension $n = 1000$.
The total number of iterations is equal to $k = 50$.
\par\vspace{\baselineskip}
\begin{center}
  \begin{minipage}{\textwidth}
  \begin{minipage}[b]{0.49\textwidth}
    \centering
    \includegraphics[width=0.95\textwidth]{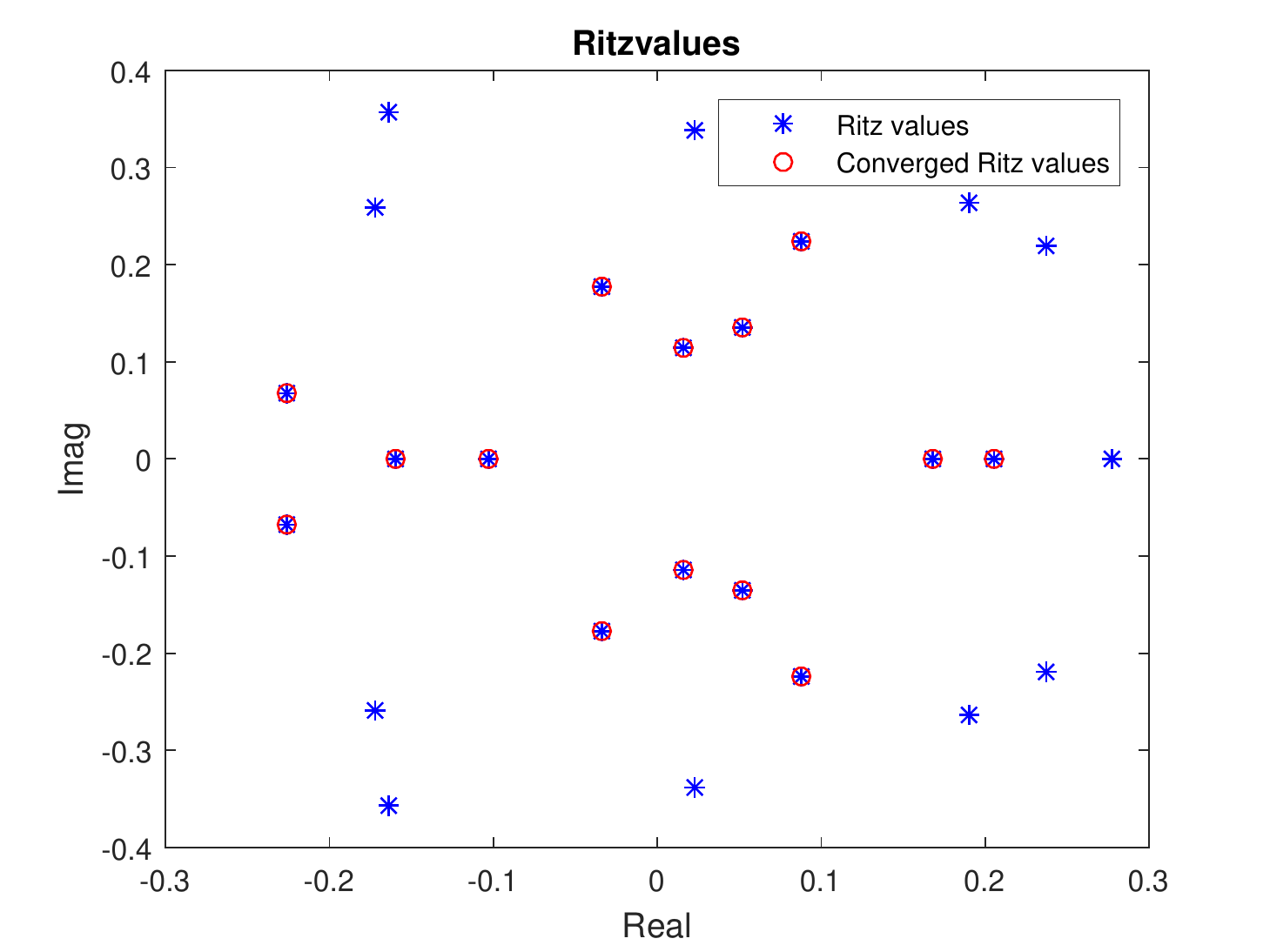}
    \captionof{figure}{Eigenvalue approximations of the infinite bi-Lanczos method
    applied to problem \eqref{eq:example1}. Circles correspond to
    approximations that have converged after $k=50$.}\label{fig:1}
  \end{minipage}
  \hfill
  \begin{minipage}[b]{0.49\textwidth}
    \centering
    \begin{tabular}{ccc}
    \hline \rule{0pt}{2.7ex}%
      $i$   & $|\theta_i^{(k_0)}|^{-1}$    & $\kappa((\theta_i^{(k_0)})^{-1})$  \\
    \hline \rule{0pt}{2.3ex}
      1     &$ 1.029\cdot 10^{-1} $     & $1.267\cdot 10^{3}  $\\
      2     &$ 1.157\cdot 10^{-1} $     & $2.510\cdot 10^{3}  $\\
      3     &$ 1.157\cdot 10^{-1} $     & $2.510\cdot 10^{3}  $\\
      4     &$ 1.440\cdot 10^{-1} $     & $1.697\cdot 10^{3}  $\\
      5     &$ 1.440\cdot 10^{-1} $     & $1.697\cdot 10^{3}  $\\
      6     &$ 1.593\cdot 10^{-1} $     & $1.846\cdot 10^{3}  $\\
      7     &$ 1.593\cdot 10^{-1} $     & $1.925\cdot 10^{3}  $\\
      8     &$ 1.803\cdot 10^{-1} $     & $7.315\cdot 10^{2}  $\\
      9     &$ 1.803\cdot 10^{-1} $     & $7.315\cdot 10^{2}  $\\ \hline
    \end{tabular}
      \captionof{table}{The condition numbers for the nine converged eigenvalues closest to zero. The values are computed using the approximate eigentriplets after $k=50$ iterations.}\label{tab:1}
    \end{minipage}
  \end{minipage}
\end{center}
Figure~\ref{fig:1} shows the approximated eigenvalues, and distinguishes the ones converged after $k=50$ iterations by a circle around them, which are obviously the ones closest to zero.
The two-sided approach has the advantage that during the process a condition number estimate is available, enabling the user to define a satisfying convergence criterion.
The condition numbers shown in Table~\ref{tab:1} correspond to the converged eigenvalues and can be computed as (cf. \cite{Tisseur:2000:BACKWARD})
\[
\kappa(\lambda,M) := \frac{\alpha \|x\|_2\|y\|_2}{ |\lambda| \left|y^*M'(\lambda)x\right|}
= \frac{\left(|\lambda|^2 \|I\|_2 + \|A_0\|_2 + |e^{-\lambda}|\|A_1\|_2\right) \|x\|_2\|y\|_2}{ |\lambda| \left|y^*(-2\lambda I - e^{-\lambda}A_1)x\right|}.
\]
We also compare the infinite bi-Lanczos method to the infinite Arnoldi method (IAR) as presented in \cite{Jarlebring:2012:INFARNOLDI}.
Figure~\ref{fig:2} shows for both methods the error in the eigenvalues against the iterations,
and Figure~\ref{fig:3} contains the error of both methods against the computation time in seconds.
The computing time illustration 
given in Figure~\ref{fig:3} depends highly on the computing environment. We
ran our simulations on several environments, including changing computer and MATLAB-version,
and observed a similar behavior in most simulations.
For the infinite bi-Lanczos method the Ritz values converge in fewer iterations,
and in general the first eigenvalue converged in less CPU-time.

\par\vspace{\baselineskip}
{
\begin{center}
  \begin{minipage}{\textwidth}
  \begin{minipage}[b]{0.49\textwidth}
    \centering
    \includegraphics[width=0.99\textwidth]{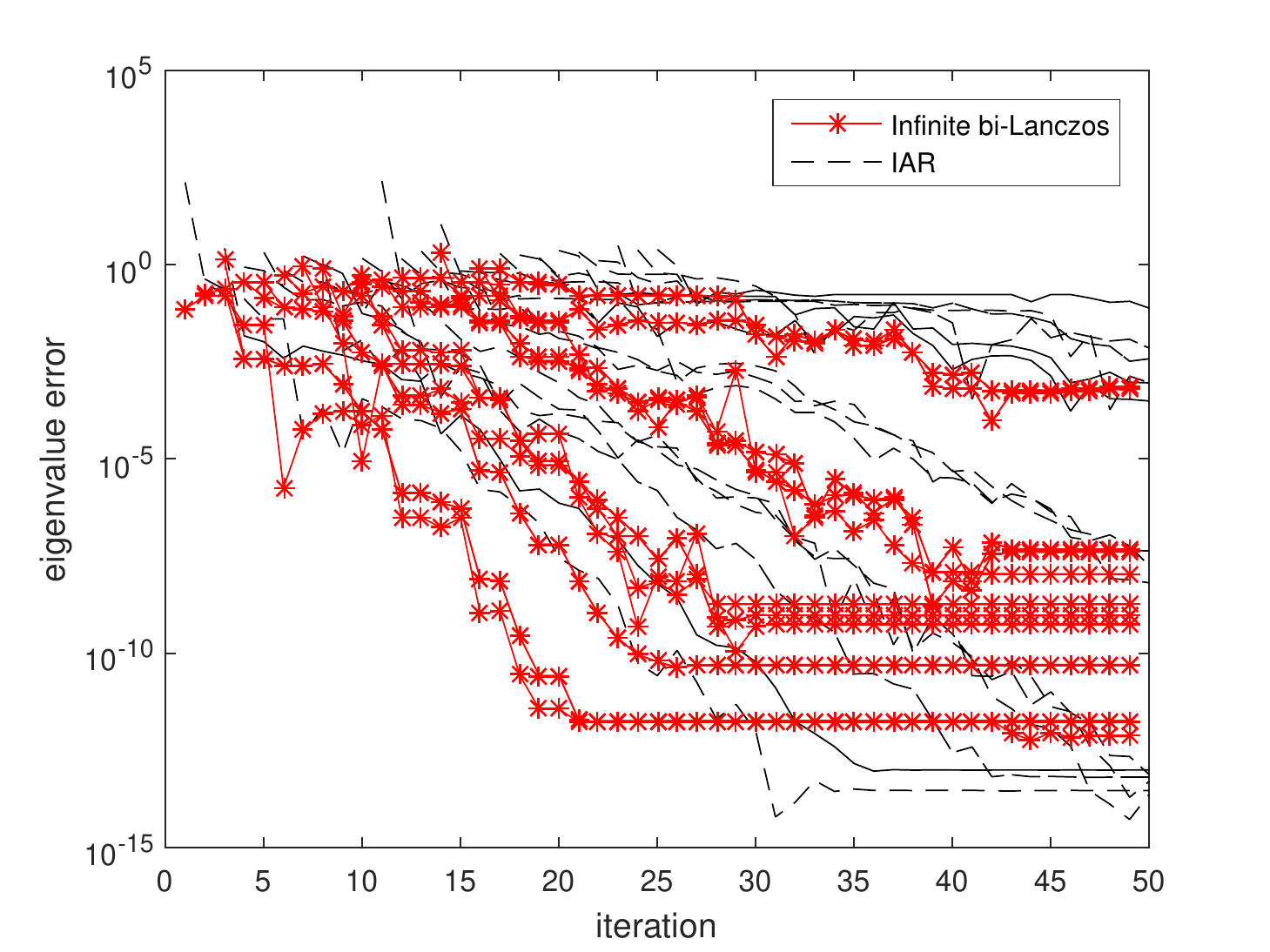}
    \captionof{figure}{Convergence diagram, eigenvalue error against the iterations.}\label{fig:2}
  \end{minipage}
  \hfill
  \begin{minipage}[b]{0.49\textwidth}
    \centering
    \includegraphics[width=0.99\textwidth]{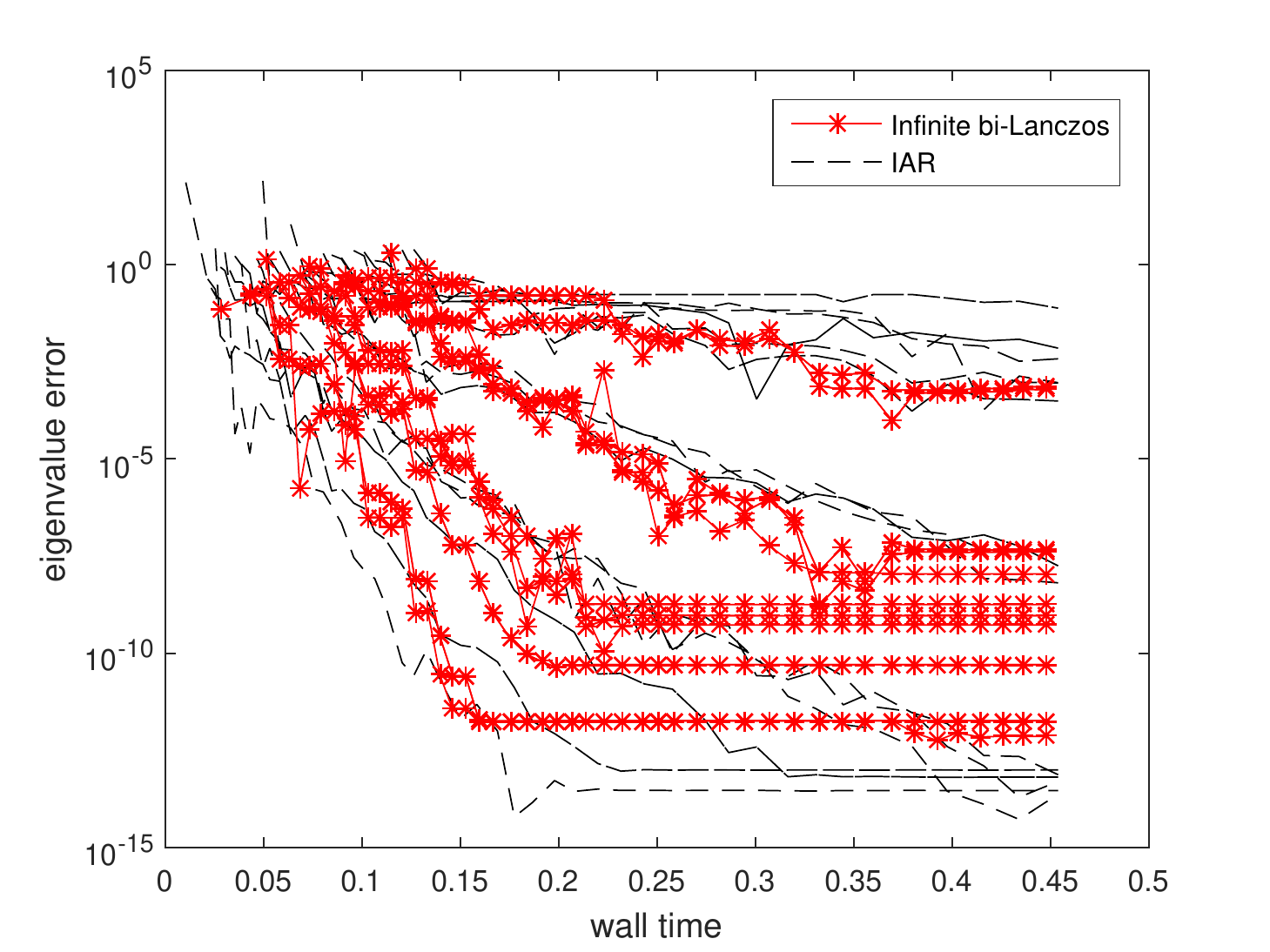}
    \captionof{figure}{Convergence diagram, eigenvalue error against the computation time (s).}\label{fig:3}
    \end{minipage}
  \end{minipage}
\end{center}
}
\noindent The faster convergence of the infinite Bi-Lanczos 
can be explained by the two subspaces that are build in the infinite bi-Lanczos.
In fact, with respect to one multiplication with $\bA$ per iteration of IAR, infinite bi-Lanczos contains per iteration a multiplication with both $\bA$ and $\bA^*$.
Because of the short recurrences the computation time of infinite bi-Lanczos can be kept decently low (and may even outperform IAR), as shown in Figure~\ref{fig:3}.
In contrast to IAR, the Bi-Lanczos procedure exhibits a stagnation in convergence. 
This is due to finite precision arithmetic. In general,
it is known that short-recurrence methods (such as the Lanczos method) are 
in general considered more sensitive 
to round-off errors than methods that use orthogonalization against all vectors in 
the basis matrix (such as the Arnoldi method).

Our implementation is based on using the computation of the scalar product as described in Remark~\ref{rem:doublesum}.
In order to illustrate the advantage of this optimization technique, 
we present computation time comparison in Table~\ref{tbl:scalarprodcomp}. Clearly, the advantage of the exploitation
of the technique Remark~\ref{rem:doublesum} has a large advantage in terms of computation time.

\begin{table}[h]
  \begin{center}
    \begin{tabular}{|c|c|c|c|c|}\hline
      &\multicolumn{2}{c|}{Using \eqref{eq:leftright_comp2}}&\multicolumn{2}{c|}{Using \eqref{eq:leftright_comp3}}  \\\hline
      &  Scal. prod. & Total & Scal. prod. & Total  \\
      \hline
      $k=10$ & $4.84\cdot 10^{-3}$ &  $8.23\cdot 10^{-2}$ & $1.58\cdot 10^{-3}$ &  $8.22\cdot 10^{-2}$\\
      $k=20$ & $9.52\cdot 10^{-3}$ &  $1.90\cdot 10^{-1}$ & $2.53\cdot 10^{-3}$ &  $1.44\cdot 10^{-1}$\\
      $k=30$ & $1.54\cdot 10^{-2}$ &  $3.67\cdot 10^{-1}$ & $3.85\cdot 10^{-3}$ &  $2.28\cdot 10^{-1}$\\
      $k=40$ & $2.11\cdot 10^{-2}$ &  $6.12\cdot 10^{-1}$ & $5.07\cdot 10^{-3}$ &  $3.24\cdot 10^{-1}$\\
      $k=50$ & $3.08\cdot 10^{-2}$ &  $9.46\cdot 10^{-1}$ & $6.08\cdot 10^{-3}$ &  $4.33\cdot 10^{-1}$\\
      $k=60$ & $4.26\cdot 10^{-2}$ &  $1.38\cdot 10^{0}$ & $7.79\cdot 10^{-3}$ &  $5.74\cdot 10^{-1}$\\
     \hline
    \end{tabular}
\caption{Computation time for the two different procedures to compute the scalar product. The timing show the accumulated CPU-time spent for $k$ iterations  in the algorithm (total) and in the computation of the scalar product (scal. prod.).
      \label{tbl:scalarprodcomp}
    }
  \end{center}
\end{table}

\subsection{A benchmark problem representing an electromagnetic cavity}
We now also consider the NEP presented in \cite{Liao:2010:NLRR} which is also
available in the collection \cite[Problem ``gun'']{Betcke:2013:NLEVPCOLL}. The problem
stems from the modelling of an electromagnetic cavity in an accelerator device.
The discretization of Maxwells equation with certain boundary conditions
leads to the NEP
\begin{equation} 
  M(\lambda)=A_0-\lambda A_1+i\sqrt{\lambda}A_2+i\sqrt{\lambda-\sigma_2^2}A_3,
\end{equation}
where $\sigma_2=108.8774$. Before applying a numerical method, 
the problem is usually shifted and scaled. We set as in $\lambda=\lambda_0+\alpha\hat{\lambda}$
where $\lambda_0=300^2$ and $\alpha=(300^2-200^2)\cdot 10$. 
This problem has been solved with a number of
methods \cite{Jarlebring:2012:INFARNOLDI,Guttel:2014:NLEIGS,VanBeeumen:2015:CORK,VanBeeumen:2015:PHD}.  We use
it as a benchmark problem to illustrate the generality and valitidy of our approach.

The convergence of the infinite Bi-Lanczos method and IAR is visualized
in Figures~\ref{fig:4} and \ref{fig:5}. 
Unlike the previous example, the convergence of infinite Bi-Lanczos does not stagnate.
For this particular choice of shift infinite Bi-Lanczos is slightly more efficient than IAR.
We carried out experiments of the algorithms for several parameter
choices, and found nothing conclusive regarding which method is more efficient in general.
Hence, if the infinite Bi-Lanczos method is favorable if also left eigenvectors are
of interest.
Moreover, the TIAR is faster, but mathematically equivalent to IAR, but often faster. 
In the same computing environment, 20 steps of TIAR requires 0.48 seconds.
Note that TIAR uses a compact tensor representation of the basis, and
therefore belongs to a slightly different class of methods. See 
\cite{VanBeeumen:2015:CORK,Kressner:2014:TOAR,Zhang:2013:MOR} for related
methods based on compact representations in various settings.

\par\vspace{\baselineskip}
{
\begin{center}
  \begin{minipage}{\textwidth}
  \begin{minipage}[b]{0.49\textwidth}
    \centering
    \includegraphics[width=0.99\textwidth]{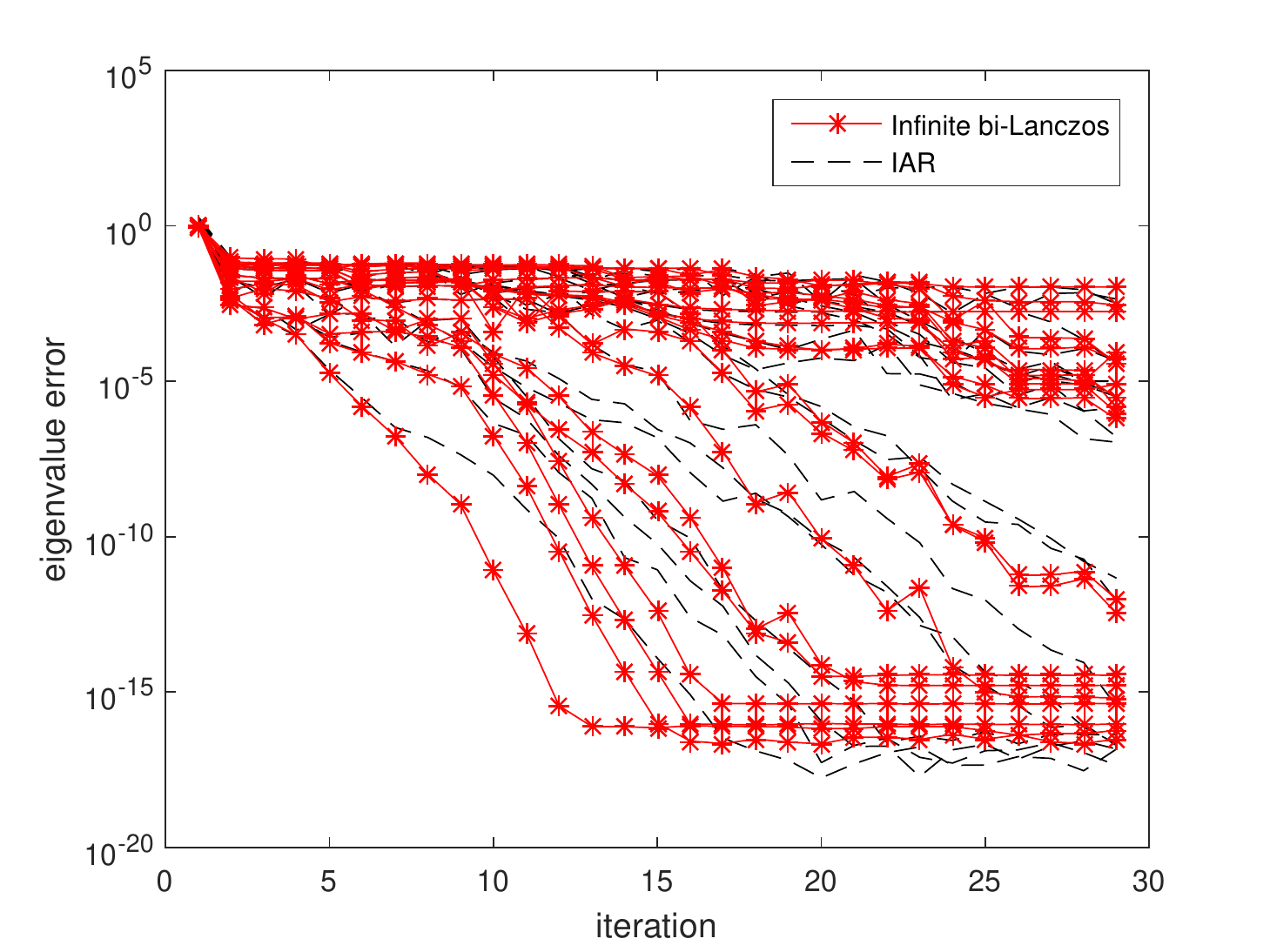}
    \captionof{figure}{Convergence diagram, eigenvalue error against the iterations.}\label{fig:4}
  \end{minipage}
  \hfill
  \begin{minipage}[b]{0.49\textwidth}
    \centering
    \includegraphics[width=0.99\textwidth]{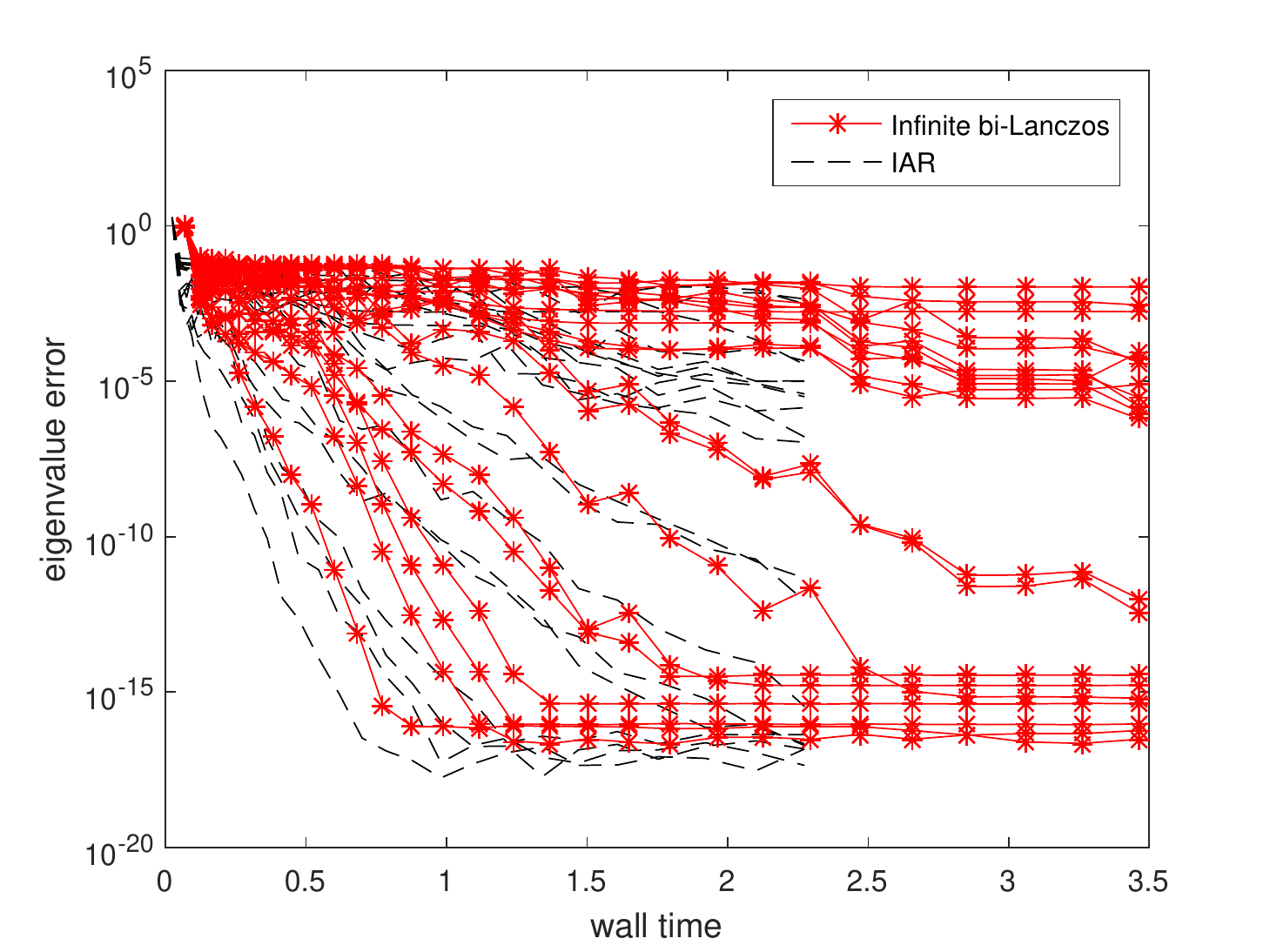}
    \captionof{figure}{Convergence diagram, eigenvalue error against the computation time (s).}\label{fig:5}
    \end{minipage}
  \end{minipage}
\end{center}
}

\section{Discussion, conclusions and outlook}
\label{sec:conclusion}
We have proposed a new two-sided Lanczos method for the nonlinear eigenvalue problem.
The method works implicitly with matrices and vectors with infinite size.
The new way of representing left type of infinite vectors is crucial to frame the two-sided method.
We intend to make the code adaptive, as the condition numbers which become available as the iterations proceed
may be used to define a satisfying convergence criterion.
We have seen that infinite bi-Lanczos can have faster convergence per iteration than the infinite Arnoldi method (IAR), which could be expected because in general two-sided methods have faster convergence (per iteration), and moreover, since infinite bi-Lanczos uses a low-term recurrence it has a lower orthogonalization cost per iteration than IAR. 

Several enhancements of IAR has been presented in the literature. Some of the developments appear to be extendable
to this Lanczos-setting, e.g., the tensor representation \cite{Jarlebring:2015:WTIARTR} and the restart techniques
\cite{Jarlebring:2014:SCHUR,Mele:2016:TIARRESTART}. Moreover, 
for certain problems it is known that a different version of IAR is more efficient, which can only be characterized
with an continuous operator (as in \cite{Jarlebring:2012:INFARNOLDI}). 
These type of adaptions are however somewhat involved due to the
fact that the left eigenvectors and the vectors representating
the left Krylov-subspace in the infinite bi-Lanczos setting is more complicated
than the right eigenvectors and subspace.

Although our approach is rigorously derived from an equivalence with the standard two-sided Lanczos method, we have provided no convergence theory. Convergence theory for standard two-sided Lanczos (for linear eigenvalue problems) is already quite involved and its specialization is certainly beyond the scope of the presented paper.
This also holds for other theory and procedures specifically designed for the standard method, such as the various possibilities to detect or overcome breakdowns (as mentioned in Section \ref{sec:standardBiLan}), and approaches to control the loss of biorthogonality (see 
\cite{Day:1997:Lanczos}). 

\section*{Acknowledgements}
We would like to thank Michiel Hochstenbach for carefully reading this paper and helpful comments on this work.
We also appreciate the encouragement of David Bindel to look at  Lanczos-type methods for nonlinear
eigenvalue problems. The second author gratefully acknowledges the support of the Swedish Research Council under Grant No.
621-2013-4640.

%
\bibliographystyle{plain}
\bibliography{eliasbib,misc}
\end{document}